\documentclass{baustms}

\citesort
\usepackage[utf8]{inputenc}

\DeclareMathOperator{\rank}{rank}
\newcommand{\Z}{\mathbb Z}

\newcommand{\Fp}{\mathbb F_p}
\newcommand{\leg}[2]{\left(\frac{#1}{#2}\right)}

\theoremstyle{cupthm}
\newtheorem{thm}{Theorem}[section]
\newtheorem{prop}[thm]{Proposition}

\newtheorem{lemma}[thm]{Lemma}
\theoremstyle{cupdefn}

\theoremstyle{cuprem}

\numberwithin{equation}{section}

\begin{document}

\runningtitle{A determinant congruence conjectured by Sun}
\title{A Determinant Congruence Conjectured by Sun}

\cauthor
\author[1]{Yutong Zhang}
\address[1]{School of Mathematics, Sichuan University, 610065 Chengdu, China
\email{yutongzhang@stu.scu.edu.cn}
ORCID: 0009-0000-1220-0702}

\author[2]{Yaoran Yang}
\address[2]{School of Mathematics, Sichuan University, 610065 Chengdu, China
\email{yangyaoran@stu.scu.edu.cn}
ORCID: 0009-0004-2832-9163}

\authorheadline{Y. Zhang and Y. Yang}

\begin{abstract}
We prove a strengthened form of a conjecture of Sun on a determinant attached to a binary quadratic form. Let $n>3$ and let $c,d\in\Z$. If $n$ is composite, then
\[
 \det\big[(i^2+cij+dj^2)^{n-2}\big]_{0\leq i,j\leq n-1}\equiv 0\pmod {n^2}
\]
with no condition on $c$ and $d$. If $n=p$ is prime, the same congruence holds whenever the Legendre symbol $\leg{d}{p}$ is $-1$. For composite $n$, a polynomial determinant is divisible by two Vandermonde factors; after specialisation, their product already yields the required square divisor. For prime $n=p$, we estimate the rank of the matrix modulo $p$. The required rank defect follows from a coefficient cancellation obtained from the involution $t\mapsto d/t$ on $\Fp^\times$ and the condition $\leg{d}{p}=-1$.
\end{abstract}


\maketitle

\section{Introduction}

For an odd positive integer $n$, write $\leg{\cdot}{n}$ for the Jacobi symbol; for prime $n$, the same notation denotes the Legendre symbol. Sun has introduced a number of determinants whose entries are defined by quadratic forms and Legendre or Jacobi symbols, and has posed several conjectures about their arithmetic behaviour; see, for example, \cite{Sun2019,Sun2024}. For any integer $n>3$, put
\begin{equation}\label{eq:Dn}
 D_n(c,d)=\det\big[(i^2+cij+dj^2)^{n-2}\big]_{0\leq i,j\leq n-1},
\end{equation}
where $c,d\in\Z$. One conjecture of Sun, \cite[Conjecture 4.1]{Sun2024}, was stated for odd $n$ and asserts that
\[
        D_n(c,d)\equiv0\pmod {n^2}
        \qquad\hbox{whenever}\qquad
        \leg{d}{n}=-1.
\]
We prove this conjecture in the following slightly stronger form: in the composite case, neither the parity of $n$ nor any Jacobi-symbol condition is needed.

Determinants of this type lie at the intersection of determinant evaluations, finite-field methods and character sums. Classical related work includes Carlitz's cyclotomic matrices \cite{Carlitz1959} and Chapman's Legendre-symbol matrices \cite{Chapman2004}. Motivated by Sun's conjectures, Krachun, Petrov, Sun and Vsemirnov \cite{KPSV2020} proved several determinant identities involving Jacobi symbols. Further results on determinants connected with quadratic residues and Legendre symbols were obtained, among others, by Wu \cite{Wu2021}, Wang and Wu \cite{WangWu2022}, Wu, She and Ni \cite{WSN2022}, Grinberg, Sun and Zhao \cite{GSZ2022}, Luo and Sun \cite{LuoSun2023}, Guo, Li, Tao and Wei \cite{GuoLiTaoWei2024}, and Wang and Sun \cite{WangSun2024}. She and Sun \cite{SheSun2025} studied several determinants and congruences associated with powers and Jacobi-symbol values of the binary form $i^2+cij+dj^2$, while Zhu and Ren \cite{ZhuRen2024} treated related prime-modulus subdeterminants. The result proved below concerns Sun's full $n\times n$ determinant \eqref{eq:Dn} with integral-power entries and establishes the modulus $n^2$ congruence in \cite[Conjecture 4.1]{Sun2024}.

The proof separates naturally into composite and prime moduli. If $n$ is composite, an alternating polynomial determinant gives a product of two Vandermonde factors, and the special value
\[
     \prod_{0\le r<s\le n-1}(s-r)
\]
is divisible by $n$. Thus its square divides $D_n(c,d)$. If $n=p$ is prime, the problem is reduced to a rank estimate over $\Fp$. After the polynomial
\[
      (X^2+cXY+dY^2)^{p-2}
\]
is reduced as a function on $\Fp^2$, its coefficient matrix has one zero row and two dependent rows. The zero row is the only delicate point; it follows from the involution $t\mapsto d/t$ and the assumption $\leg{d}{p}=-1$. A Smith normal form argument then lifts the rank defect modulo $p$ to divisibility by $p^2$ over the integers.

The main result is the following.

\begin{thm}\label{thm:main}
Let $n>3$ be an integer, and let $c,d\in\Z$. Suppose either that $n$ is composite, or that $n$ is prime and the Legendre symbol $\leg{d}{n}$ is $-1$. Then
\[
\det\big[(i^2+cij+dj^2)^{n-2}\big]_{0\leq i,j\leq n-1}
\equiv 0\pmod {n^2}.
\]
\end{thm}

\section{The composite case}

We start with a standard divisibility fact, included to make the argument over $\Z$ explicit.

\begin{lemma}\label{lem:vandermonde-divisibility}
Let $R$ be a commutative ring, and let
\[
G(x_0,\ldots,x_{N-1})\in R[x_0,\ldots,x_{N-1}]
\]
be alternating in the following sense: for every $r<s$, the polynomial obtained from $G$ by setting $x_s=x_r$ is zero. Then
\[
       \prod_{0\le r<s\le N-1}(x_s-x_r)
\]
divides $G$ in $R[x_0,\ldots,x_{N-1}]$.
\end{lemma}

\begin{proof}
We argue by induction on $N$. The assertion is clear for $N\leq1$. Put $z=x_{N-1}$ and
\[
       S=R[x_0,\ldots,x_{N-2}].
\]
We first prove that
\[
       P(z):=\prod_{i=0}^{N-2}(z-x_i)
\]
divides $G$ in $S[z]$. Suppose that, for some $0\leq k\leq N-2$, we have already written
\[
       G=\prod_{i=0}^{k-1}(z-x_i)G_k
       \qquad(G_k\in S[z]),
\]
with the empty product understood as $1$. Since $G$ is zero after the specialisation $z=x_k$, we obtain
\[
       0=\prod_{i=0}^{k-1}(x_k-x_i)\,G_k|_{z=x_k}.
\]
The product $\prod_{i=0}^{k-1}(x_k-x_i)$ is monic as a polynomial in $x_k$, and hence is a non-zero-divisor. Therefore $G_k|_{z=x_k}=0$, so the monic factor theorem gives $G_k=(z-x_k)G_{k+1}$ for some $G_{k+1}\in S[z]$. Repeating this for $k=0,1,\ldots,N-2$, we obtain
\[
       G=P(z)H
       \qquad(H\in S[z]).
\]

It remains to apply the induction hypothesis to $H$ in the variables $x_0,\ldots,x_{N-2}$. If $0\leq r<s\leq N-2$ and we set $x_s=x_r$, then the corresponding specialisation of $G$ is zero. Since the corresponding specialisation of $P(z)$ is monic in $z$, it is a non-zero-divisor; hence the corresponding specialisation of $H$ is zero. Thus $H$ is alternating in $x_0,\ldots,x_{N-2}$ in the same sense. By induction,
\[
       \prod_{0\le r<s\le N-2}(x_s-x_r)
\]
divides $H$. Multiplying this divisor by $P(z)$ gives the required Vandermonde product for $x_0,\ldots,x_{N-1}$.
\end{proof}

\begin{prop}\label{prop:composite}
Let $n>3$ be a composite integer. Then, for all $c,d\in\Z$,
\[
D_n(c,d)\equiv 0\pmod {n^2}.
\]
\end{prop}

\begin{proof}
Put
\[
F(X,Y)=(X^2+cXY+dY^2)^{n-2}
\]
and consider the polynomial determinant
\[
\Phi(\mathbf x,\mathbf y)=
\det\big[F(x_i,y_j)\big]_{0\leq i,j\leq n-1}
\in \Z[x_0,\ldots,x_{n-1},y_0,\ldots,y_{n-1}].
\]
If two $x$-variables are made equal, then the corresponding two rows are equal; hence $\Phi$ is alternating in $x_0,\ldots,x_{n-1}$ in the sense of Lemma \ref{lem:vandermonde-divisibility}. Applying that lemma over $\Z[y_0,\ldots,y_{n-1}]$, we find that
\[
       V(\mathbf x):=\prod_{0\leq r<s\leq n-1}(x_s-x_r)
\]
divides $\Phi$. Thus
\[
       \Phi(\mathbf x,\mathbf y)=V(\mathbf x)\Phi_1(\mathbf x,\mathbf y)
\]
for some
\[
       \Phi_1\in\Z[x_0,\ldots,x_{n-1},y_0,\ldots,y_{n-1}].
\]
If two $y$-variables are made equal, then the corresponding two columns of $\Phi$ are equal. Since $V(\mathbf x)$ is independent of the $y$-variables and non-zero in the integral domain under consideration, cancellation shows that $\Phi_1$ is alternating in $y_0,\ldots,y_{n-1}$ in the same sense. Applying Lemma \ref{lem:vandermonde-divisibility} again, now over $\Z[x_0,\ldots,x_{n-1}]$, gives
\[
\Phi(\mathbf x,\mathbf y)=V(\mathbf x)V(\mathbf y)H(\mathbf x,\mathbf y)
\]
for some $H\in\Z[x_0,\ldots,x_{n-1},y_0,\ldots,y_{n-1}]$, where
\[
       V(\mathbf y)=\prod_{0\leq r<s\leq n-1}(y_s-y_r).
\]
Specialising $x_i=y_i=i$ yields
\begin{equation}\label{eq:Vn-factor}
D_n(c,d)=V_n^2\,H(0,1,\ldots,n-1,0,1,\ldots,n-1),
\end{equation}
where
\[
V_n=\prod_{0\leq r<s\leq n-1}(s-r)=\prod_{k=1}^{n-1}k^{\,n-k}.
\]
It remains to prove that $n\mid V_n$.

Let $p^\alpha\parallel n$. The difference $s-r=p$ occurs for exactly $n-p$ pairs $(r,s)$ with $0\le r<s\le n-1$, and hence
\[
       \nu_p(V_n)\ge n-p.
\]
We claim that $n-p\ge\alpha$. It suffices to show that the set
\[
       \{p,p+1,p+2,\ldots,n\}
\]
has at least $\alpha+1$ elements. Indeed, the $\alpha+1$ elements
\[
       p,\ p+1,\ p^2,\ p^3,\ldots,\ p^\alpha
\]
with the evident omission of $p^2,\ldots,p^\alpha$ when $\alpha=1$, are distinct elements of this set. Here $p^\alpha\leq n$; and, when $\alpha=1$, the compositeness of $n$ and the divisibility $p\mid n$ give $n\ge2p\ge p+1$. Thus $n-p+1\ge\alpha+1$, proving the claim. Therefore $\nu_p(V_n)\ge\alpha$ for every prime power $p^\alpha\parallel n$. Hence $n\mid V_n$, and \eqref{eq:Vn-factor} implies $n^2\mid D_n(c,d)$.
\end{proof}

\section{The prime case}

We shall use the following elementary consequence of the Smith normal form.

\begin{lemma}\label{lem:smith}
Let $A$ be an $N\times N$ integer matrix, let $p$ be a prime, and let $r\geq0$. If
\[
        \rank_{\Fp}(A\bmod p)\le N-r,
\]
then $p^r$ divides $\det A$.
\end{lemma}

\begin{proof}
Let
\[
        UAV=\operatorname{diag}(s_1,\ldots,s_N)
\]
be a Smith normal form of $A$, where $U,V$ are unimodular integer matrices. The reductions of $U$ and $V$ modulo $p$ are invertible over $\Fp$, and hence multiplication by them preserves rank over $\Fp$. Thus the rank of $A$ modulo $p$ is the number of diagonal entries $s_i$ whose residue modulo $p$ is non-zero. If this rank is at most $N-r$, then at least $r$ of the $s_i$ are divisible by $p$. Consequently
\[
        p^r\mid s_1\cdots s_N=\det(UAV)=\pm\det A,
\]
as required. If some $s_i=0$, the conclusion is of course still valid.
\end{proof}

\begin{lemma}\label{lem:critical}
Let $p\geq3$ be a prime, let $c,d\in\Z$, and suppose that $\leg{d}{p}=-1$. In $\Fp[T]$, write
\begin{equation}\label{eq:alpha-def}
(T^2+cT+d)^{p-2}=\sum_{a=0}^{2p-4}\alpha_a T^a.
\end{equation}
If $m=(p-3)/2$, then
\[
        \alpha_m+\alpha_{m+p-1}=0
        \qquad\hbox{in }\Fp.
\]
\end{lemma}

\begin{proof}
We still write $c$ and $d$ for their residues modulo $p$, and put
\[
        f(T)=T^2+cT+d\in\Fp[T].
\]
The assumption $\leg{d}{p}=-1$ implies $d\ne0$ in $\Fp$. Define
\[
        S=\sum_{t\in\Fp^\times} f(t)^{p-2}t^{-m}.
\]
Here and below, negative powers are interpreted in the multiplicative group $\Fp^\times$. By \eqref{eq:alpha-def},
\[
S=\sum_{a=0}^{2p-4}\alpha_a
      \sum_{t\in\Fp^\times}t^{a-m}.
\]
For any integer $k$, the usual orthogonality relation on $\Fp^\times$ gives
\[
\sum_{t\in\Fp^\times}t^k=
 \begin{cases}
 -1,& p-1\mid k,\\
 0,& p-1\nmid k
 \end{cases}
\qquad\hbox{in }\Fp.
\]
Among the integers $a$ with $0\le a\le 2p-4$, the only ones congruent to $m$ modulo $p-1$ are $a=m$ and $a=m+p-1$. Hence
\begin{equation}\label{eq:S-alpha}
        S=-(\alpha_m+\alpha_{m+p-1}).
\end{equation}

Now the map $t\mapsto d/t$ is a permutation of $\Fp^\times$. Moreover
\[
        f(d/t)=d\,t^{-2}f(t)
\]
for every $t\in\Fp^\times$. Therefore
\begin{align*}
 f(d/t)^{p-2}(d/t)^{-m}
 &= (d\,t^{-2})^{p-2}f(t)^{p-2}d^{-m}t^m  \\
 &= d^{-m-1}t^{m+2}f(t)^{p-2}.
\end{align*}
Since
\[
        m+2=(p+1)/2\equiv -m\pmod {p-1}
\]
and, by Euler's criterion,
\[
        d^{-m-1}
        =d^{-(p-1)/2}
        =\bigl(d^{(p-1)/2}\bigr)^{-1}
        =\leg{d}{p}^{-1}
        =-1
        \quad\hbox{in }\Fp,
\]
we obtain
\[
        f(d/t)^{p-2}(d/t)^{-m}
        =-f(t)^{p-2}t^{-m}
\]
for all $t\in\Fp^\times$. Summing over $t\in\Fp^\times$ and using that $t\mapsto d/t$ is a permutation gives $S=-S$. Since $p$ is odd, $S=0$. Equation \eqref{eq:S-alpha} now gives the desired cancellation.
\end{proof}

\begin{lemma}\label{lem:coefficient-matrix}
Let $p>3$ be a prime, let $c,d\in\Z$, and suppose that $\leg{d}{p}=-1$. Let $C=(C_{r,s})_{0\le r,s\le p-1}$ be the coefficient matrix of the unique polynomial
\[
        R(X,Y)=\sum_{0\le r,s\le p-1}C_{r,s}X^rY^s\in\Fp[X,Y]
\]
which has degree at most $p-1$ in each variable and represents the same function on $\Fp^2$ as
\[
        (X^2+cXY+dY^2)^{p-2}.
\]
Then
\[
        \rank_{\Fp}C\le p-2.
\]
\end{lemma}

\begin{proof}
First note the uniqueness in the statement. If a polynomial of degree at most $p-1$ in each variable vanishes on $\Fp^2$, then, after fixing $Y=y\in\Fp$, it is a polynomial in $X$ of degree at most $p-1$ with $p$ roots; hence it is zero. Its coefficients, as polynomials in $Y$, again have degree at most $p-1$ and vanish at all $y\in\Fp$, and therefore are zero. Existence follows by applying one-variable Lagrange interpolation successively in the two variables.

Let the coefficients $\alpha_a$ be defined by \eqref{eq:alpha-def}. We have the homogeneous expansion
\begin{equation}\label{eq:homogeneous-expansion}
        (X^2+cXY+dY^2)^{p-2}
        =\sum_{a=0}^{2p-4}\alpha_a X^aY^{2p-4-a}.
\end{equation}
For $e\ge0$ put $\bar e=0$ if $e=0$, and otherwise let $\bar e$ be the unique integer in $\{1,\ldots,p-1\}$ congruent to $e$ modulo $p-1$. Then $z^e=z^{\bar e}$ for all $z\in\Fp$. Since every monomial in \eqref{eq:homogeneous-expansion} has total degree $2p-4<2p-2$, the two exponents in the same monomial cannot both be at least $p$. Thus at most one exponent in each monomial needs reduction to transform the right hand side of \eqref{eq:homogeneous-expansion} into $R(X,Y)$. This is the point at which we use $p>3$; for example, the reduction of $Y^{2p-4}$ to $Y^{p-3}$ requires $p-3\in\{1,\ldots,p-1\}$.

We now record the reduction explicitly. The term $a=0$ gives $\alpha_0Y^{p-3}$. For $1\le r\le p-4$, the two indices $a=r$ and $a=r+p-1$ both reduce to the same monomial
\[
        X^rY^{p-3-r}.
\]
The three boundary indices $a=p-3,p-2,p-1$ give respectively
\[
        \alpha_{p-3}X^{p-3}Y^{p-1},\qquad
        \alpha_{p-2}X^{p-2}Y^{p-2},\qquad
        \alpha_{p-1}X^{p-1}Y^{p-3}.
\]
Finally, since $\alpha_{2p-4}=1$, the index $a=2p-4$ gives $X^{p-3}$. Hence the same function on $\Fp^2$ is represented by
\begin{align}\label{eq:R-poly}
R(X,Y)={}&\alpha_0Y^{p-3}
 +\sum_{r=1}^{p-4}(\alpha_r+\alpha_{r+p-1})X^rY^{p-3-r}
 +X^{p-3}  \nonumber\\
&+\alpha_{p-3}X^{p-3}Y^{p-1}
 +\alpha_{p-2}X^{p-2}Y^{p-2}
 +\alpha_{p-1}X^{p-1}Y^{p-3}.
\end{align}

Let $m=(p-3)/2$. Since $p>3$, we have $1\le m\le p-4$. In \eqref{eq:R-poly}, the only possible non-zero entry in row $m$ of $C$ is
\[
        C_{m,p-3-m}=\alpha_m+\alpha_{m+p-1},
\]
and this is zero by Lemma \ref{lem:critical}. Hence row $m$ is zero. In addition, row $0$ and row $p-1$ can have non-zero entries only in column $p-3$, so these two rows are linearly dependent. These rows are distinct from row $m$. Consequently the row rank of $C$ is at most $p-2$.
\end{proof}

\begin{prop}\label{prop:prime}
Let $p>3$ be a prime, let $c,d\in\Z$, and suppose that $\leg{d}{p}=-1$. Then
\[
        D_p(c,d)\equiv0\pmod {p^2}.
\]
\end{prop}

\begin{proof}
Let $a_i$ denote the residue class of $i$ in $\Fp$. Let
\[
        M=\big[(a_i^2+ca_ia_j+da_j^2)^{p-2}\big]_{0\le i,j\le p-1}
\]
be the reduction modulo $p$ of the integer matrix defining $D_p(c,d)$. Here $c$ and $d$ are also viewed modulo $p$. By Lemma \ref{lem:smith}, it is enough to prove that
\[
        \rank_{\Fp}M\le p-2.
\]

Let $R(X,Y)$ and $C=(C_{r,s})$ be as in Lemma \ref{lem:coefficient-matrix}. Since $R$ represents the same function on $\Fp^2$ as $(X^2+cXY+dY^2)^{p-2}$, we have
\[
        M=\big[R(a_i,a_j)\big]_{0\le i,j\le p-1}.
\]
Set
\[
        V=(a_i^r)_{0\le i,r\le p-1}.
\]
Then
\[
        M=VCV^{\mathrm t}.
\]
Hence
\[
        \rank_{\Fp}M\leq\rank_{\Fp}C\le p-2
\]
by Lemma \ref{lem:coefficient-matrix}. Lemma \ref{lem:smith}, with $N=p$ and $r=2$, now implies $p^2\mid D_p(c,d)$.
\end{proof}

\begin{proof}[Proof of Theorem \ref{thm:main}]
If $n$ is composite, the assertion follows from Proposition \ref{prop:composite}. If $n$ is prime, then the assumption in the theorem is precisely $\leg{d}{n}=-1$ as a Legendre-symbol condition, and Proposition \ref{prop:prime} applies.
\end{proof}
\section*{Declaration of Generative AI and AI-Assisted Technologies in the Writing Process}
During the preparation of this work, the authors used DeepSeek to build a specialized agent for solving mathematical problems, which was employed to generate an initial proof of the main theorem. After using this tool, the authors reviewed and edited the content as needed and take full responsibility for the content of the published article.

\end{document}